\documentclass[11pt]{article}
\usepackage[latin1]{inputenc}
\usepackage[T1]{fontenc}
\usepackage{amsmath,amssymb,amstext}
\usepackage{theorem}
\usepackage{a4wide}
\usepackage{graphicx,color}
\usepackage{multicol}
\usepackage[english]{babel}
\usepackage{enumerate}
\usepackage{eufrak}

\newtheorem{theorem}{Theorem}

\newtheorem{lemma}[theorem]{Lemma}

\newtheorem{proposition}[theorem]{Proposition}

\newenvironment{proof}[1][Proof]{\textbf{#1.} }{\ \rule{0.5em}{0.5em}}

\renewcommand{\geq}{\geqslant}
\def\leq{\leqslant}

\def\1{{\mathbf{1}}}

\def\1{{\mathbf{1}}}
\def\0.5{{\frac{1}{2}}}

\renewcommand{\thefootnote}{\fnsymbol{footnote}}
\begin{document}

\begin{center}
{\Large \textbf{Volatility estimation in fractional Ornstein-Uhlenbeck models }} \\[0pt]
~\\[0pt]
Salwa Bajja\footnote{%
National School of Applied Sciences - Marrakesh, Cadi Ayyad University,
Marrakesh, Morocco. Email: \texttt{salwa.bajja@gmail.com }}, Khalifa
Es-Sebaiy\footnote{%
Department of Mathematics, Kuwait University, Kuwait. Email: \texttt{%
khalifasbai@gmail.com}} and Lauri Viitasaari \footnote{%
Department of Mathematics and Statistics, University of Helsinki, Helsinki, P.O. Box
68, FIN-00014 University of Helsinki, Finland. E-mail:\texttt{%
lauri.viitasaari@iki.fi}}\\[0pt]
\textit{Cadi Ayyad University, Kuwait University and University of Helsinki}\\[0pt]
~\\[0pt]
\end{center} 

\begin{abstract}
\medskip In this article we study the asymptotic behaviour of the
realized quadratic variation of a process $\int_{0}^{t}u_{s}dY_{s}^{(1)}$%
, where $u$ is a $\beta$-H\"older continuous process with $\beta > 1-H$
 and $Y_{t}^{(1)}=\int_{0}^{t}e^{-s}dB^{H}_{a_s}$, where $a_{t}=He^{\frac{t%
}{H}} $ and $B^H$ is a fractional Brownian motion, is connected to the fractional Ornstein-Uhlenbeck process of the second kind. We prove almost sure convergence uniformly in time, and a stable weak convergence for the realized quadratic variation. As an application, we construct strongly consistent estimator for the integrated volatility parameter in a model driven by $Y^{(1)}$. 
\end{abstract}

\renewcommand{\thefootnote}{\arabic{footnote}} \noindent {\textbf{2010 Mathematics Subject Classification:}}
60G15, 60G22, 62M09, 62F12.
\renewcommand{\thefootnote}{\arabic{footnote}}

\renewcommand{\thefootnote}{\arabic{footnote}} \noindent {\textbf{Keywords:}}
Fractional Brownian motion; Quadratic variation; Stable convergence; Volatility.

\section{Introduction}
The realized quadratic variation is a powerful tool in the statistical analysis of stochastic processes, and it has received a lot of attention in the literature. Furthermore, its generalization, the realized power variation of order $p>0$, have received similar attention as it can tackle with several problems related to realized quadratic variation. For example, the asymptotic normality does not hold for realized quadratic variation in the case of the fractional Brownian motion $B^H$ with $H>\frac34$, while asymptotic normality hold for realized power variation if one chooses $p$ large enough. 

The realized power variation of order $p$ (quadratic variation if $p=2$) is defined as
\begin{equation}  \label{eq:statistic}
\sum_{i=1}^{[nt]}\left|X_{i/n}-X_{(i-1)/n}\right|^p
\end{equation}
where $\{X_{t},t\geq0\}$ is a stochastic process. It was originally introduced in
Barndorff-Nielsen and Shephard (\cite{BS2002}, \cite{BS2003},\cite{BS2004a},%
\cite{BS2004b}) to estimate the integrated volatility in some stochastic
volatility models used in quantitative finance and also, under an
appropriate modification, to estimate the jumps of the processes. The main interest in the mentioned papers is the asymptotic behaviour of
appropriately normalised version of the statistic (\ref{eq:statistic}), when the process $X_{t}$ is a stochastic
integral with respect to a Brownian motion. Refinements of the results
have been obtained in \cite{W2003} and \cite{W2005}, and further extensions
can be found in \cite{BSal}. 

The asymptotic behaviour of the power variation of a stochastic integral $Z_{t}=\int_{0}^{t}u_{s}dB_{s}^{H}$ with respect to a fractional Brownian motion was studied in \cite{CNW}. In \cite{CNW} the authors proved that if $u=\{u_{t},t\geq0\}$ has finite $q$-variation for some $q<1/(1-H)$, then
\begin{eqnarray}
n^{-1+pH}V_{p}^{n}(Z)_{t} &\longrightarrow& c_{1,p}\int_{0}^{t}|u_{s}|^{p}ds
\end{eqnarray}
uniformly in probability in any compact sets of t, where $c_{1,p}=\mathrm{I%
\kern-0.16em E}[|B_{1}^{H}|^{p}].$ The authors also proved central limit theorem for $H \in (0,\frac{3}{4}]$. However, the
condition $H\in(0,\frac{3}{4}]$ is critical in \cite{CNW}. The first
objective of \cite{HNZ} was to remove this
restriction. They used higher order differences and defined the power variation as $V_{k,p}^{n}(Z)_{t}=\sum_{i=1}^{[nt]-k+1}\left|%
\sum_{j=0}^{k}(-1)^{k-j}C_{j}^{k}Z_{(i+j-1)/n}\right|^{p}$ for certain numbers $C_j^k$.

In this paper we study the asymptotic behaviour of the realized quadratic variation of a process of the form $
\int_{0}^{t}u_{s}dY^{(1)}_{s}$, where $Y_{t}^{(1)}=%
\int_{0}^{t}e^{-s}dB_{a_{s}}$; $a_{t}=He^{t/H}$, $B^{H}$ is a fractional
Brownian motion with Hurst parameter $H\in(0,1)$, and $u$ is a $\beta$-H\"older continuous process with $\beta > 1-H$. We note that the process $Y^{(1)}$ is connected to the fractional Ornstein-Uhlenbeck process of the second kind, that is defined through the Lamperti transform of the fractional Brownian motion. Equivalently, fractional Ornstein-Uhlenbeck process of the second kind can be defined as the solution to the stochastic differential equation 
\begin{eqnarray} 
dX_{t} &=& -\theta X_{t}dt+\sigma dY^{(1)}_{t}.
\end{eqnarray}
This process shares path properties with the fractional Brownian motion, but unlike the fractional Brownian motion with $H>\frac12$, it has a short memory. This allows one to model short memory with a transformation of the fractional Brownian motion, and thus makes the process $Y^{(1)}$ interesting. 

As our main result, we obtain almost sure and uniform convergence. In comparison, \cite{CNW} obtained uniform convergence in probability. Actually, with our modified proofs one is able to prove stronger almost sure convergence also in their case. That is, in the case of the fractional Ornstein-Uhlenbeck process of the first kind. We also establish weak convergence result provided that $H\in\left(0,\frac34\right)$.

The rest of the paper is organised as follows. In Section \ref{sec:qv} we present and proof our main results. We begin Section \ref{sec:qv} by recalling some preliminary results. In Section \ref{sec:vol}, we apply our main results to the estimation of the integrated volatility. We end the paper with a short discussions. 
\section{Main results}
\label{sec:qv}
In this section we state and prove our main results that will be applied to the estimation of integrated volatility in Section \ref{sec:vol}. We begin by recalling some preliminaries.

Suppose that $B^H=\{B^H_{t},t\geq 0\}$ is an fBm with Hurst parameter $
H\in(0,1)$. That is, $
B^H$ is a zero-mean Gaussian process with covariance function
\begin{eqnarray}  \label{covariance function}
E\left[B_{s}^{H}B_{t}^{H}\right] &=&\frac{1}{2}\left[s^{2H}+t^{2H}-|t-s|^{2H}%
\right], \ \ s,t\geq0.
\end{eqnarray}
It is well-known that, for any $\epsilon>0$, the trajectories of $B^H$ are $(H-\varepsilon)$-Hölder continuous on any
finite interval. Indeed, this is a simple consequence of the Kolmogorov continuity criterion. 

We also recall that, for $p>0$, the $p$-variation of a real-valued function $f$ on an interval
[a,b] is defined as
\begin{equation}  \label{p-variation}
var_{p}(f;[a,b])=\sup_{\pi}\left(\sum_{i=1}^{n}|f(t_{i})-f(t_{i-1})|^p%
\right)^{1/p},
\end{equation}
where the supremum is taken over all partitions $\pi=\{a=t_{0}<t_{1}<...<t_{n}=b%
\}.$ We say that $f$ has finite $p$-variation (over the interval $[a,b]$), if $var_p(f;[a,b]) < \infty$. Young proved that the integral $\int_{a}^{b}fdg$
exists as a Riemann-Stieltjes integral provided that $f$ and $g$ have finite $p$-variation and $q$-variation with $1/p+1/q>1$. Moreover, the following inequality
holds:
\begin{equation}  \label{Young inequality}
\left|\int_{a}^{b}fdg-f(a)(g(b)-g(a))\right|\leqslant
c_{p,q}var_{p}(f;[a,b])var_{q}(g;[a,b]),
\end{equation}
where $c_{p,q}=\zeta(1/q+1/p)$ , with $\zeta(s)=\sum_{n\geq1}n^{-s}.$

We denote by
\begin{equation*}
\parallel f\parallel_{\alpha}:=\sup_{a\leqslant s<t\leqslant b}\frac{%
|f(t)-f(s)|}{|t-s|^\alpha}
\end{equation*}
the H\"older seminorm of order $\alpha$. Clearly, if $f$ is $\alpha$-H\"older continuous, then it has finite $
(1/\alpha)$-variation on any finite interval. In this case we have, for any $p\geq \frac{1}{\alpha}$, that
\begin{equation}
\label{eq:holder_var}
var_{p}(f;[a,b]) \leq \parallel f\parallel_{\alpha}(b-a)^\alpha.
\end{equation}
Throughout the paper, we also assume that $T< \infty$ is fixed. That is, we consider stochastic processes on some compact interval. We denote by $\|.\|_{\infty}$ the supremum norm on $[0,T]$.

For any natural number $n\geq1$, and for any stochastic process $%
Z=\{Z_{t},t\geq0\}$, we write
\begin{eqnarray}  \label{quadratic variation of integral}
V_{n}(Z)_{t}=\sum_{i=1}^{[nt]}\left|Z_{\frac{i}{n}}-Z_{\frac{i-1}{n}}\right|^2. 
\end{eqnarray}
We will use the following two general results, taken from \cite{Lauri Viitasaari}, on the convergence of the quadratic variations of a Gaussian process.
\begin{theorem}
(\cite[Theorem 3.1]{Lauri Viitasaari}) \label{theorem:BE_bound_QV} Let $X$
be a continuous Gaussian process and denote by $V_n^X$ its quadratic
variation defined by
\begin{equation*}
V_n^X = \sum_{k=1}^n \left[\left(\Delta_k X\right)^2 - \mathrm{I\kern-0.16em
E}\left(\Delta_k X\right)^2 \right],
\end{equation*}
where $\Delta_k X = X_{\frac{k}{n}} - X_{\frac{k-1}{n}}$. Assume that
\begin{eqnarray}  \label{eq: cvge prob}
\max_{1\leqslant j \leqslant N(\pi_{n})-1}\sum_{k=1}^{N(\pi_n)-1}\frac{1}{%
\sqrt{\phi(\Delta t_{k})\phi(\Delta t_{j})}}|\mathrm{I\kern-0.16em E}%
[(X_{t_{k}}-X_{t_{k-1}})(X_{t_{j}}-X_{t_{j-1}})]| &\leqslant& H(|\pi_{n}|)
\notag
\end{eqnarray}
for some function $\phi$ and $H(|\pi_{n}|)$.\newline
If $H(|\pi_{n}|)\rightarrow 0 $ as $|\pi_{n}|$ tends to zero, then the
convergence
\begin{eqnarray}
\left|\sum_{k=1}^{N(\pi_{n})-1}\frac{(X_{t_{k}}-X_{t_{k-1}})^2}{%
\phi(t_{k}-t_{k-1})}-\sum_{k=1}^{N(\pi_{n})-1}\frac{\mathrm{I\kern-0.16em E}%
(X_{t_{k}}-X_{t_{k-1}})^2}{\phi(t_{k}-t_{k-1})}\right| &\rightarrow& 0
\end{eqnarray}
holds in probability. Furthermore, the convergence holds almost surely provided that $H(|\pi_{n}|)=\circ(\frac{1}{\log(n)})$.
\end{theorem}

The following lemma gives easy way to compute the function $H(n)$ (see \cite[Theorem 3.3]{Lauri
Viitasaari}).

\begin{lemma}[\protect\cite{Lauri Viitasaari}]
\label{lemma:lemma of Lauri} \label{lma:rate} Let $X$ be a continuous
Gaussian process such that the function $d(s,t)=E(X_{t}-X_{s})^{2}$ is in $%
C^{1,1}$ outside diagonal. Furthermore, assume that
\begin{equation}
\label{eq:derivative_cond_needed}
|\partial_{st} d(s,t)|=O\left(|t-s|^{2H-2}\right)
\end{equation}
for some $H\in (0,1), H \neq \frac12$. Then
\begin{equation*}
\max_{1\leqslant j \leqslant n} \sum_{k=1}^n \left\vert\mathrm{I\kern-0.16em
E}(\Delta_k X \Delta_j X)\right\vert \leqslant \max_{1\leqslant j \leqslant
n} d\left(\frac{j}{n},\frac{j-1}{n}\right)+ \left(\frac{1}{n}%
\right)^{1\wedge 2H}.
\end{equation*}
\end{lemma}

Finally, in order to study stable convergence in law we recall the following general convergence result taken from \cite{CNP}.
\begin{theorem}[\protect\cite{CNP}]
\label{theorem: Hypotheses} Let $(\Omega,\mathcal{F},P)$ be a complete
probability space. Fix a time interval $[0,T]$ and consider a double
sequence of random variables $\xi=\{\xi_{i,m},m\in Z_{+},1\leqslant i
\leqslant [mT]\}.$ Assume the double sequence $\xi$ satisfies the following
hypotheses.\newline
\newline
(H1) Denote $g_{m}(t):=\sum_{i=1}^{[mt]}\xi_{i,m}$. The finite dimensional
distributions of the sequence of processes $\{g_{m}(t),t \in [0,T]\}$
converges $\mathcal{F}$-stably to those of $\{B(t), t\in [0,T]\}$ as $%
m\rightarrow\infty$, where $\{B(t), t\in [0,T]\}$ is a standard Brownian
motion independent of $\mathcal{F}$.\newline
\newline
(H2) $\xi$ satisfies the tightness condition $\mathrm{I\kern-0.16em E}%
\left|\sum_{i=j+1}^{k}\xi_{i,m}\right|^4\leqslant C \left(\frac{k-j}{m}%
\right)^2$ for any $1\leqslant j \leqslant k\leqslant [mT]$.\newline
\newline
If $\{f(t), t\in [0,T]\}$ is an $\alpha-$H\"older continuous process with $%
\alpha>1/2$ and we set $X_{m}(t):=\sum_{i=1}^{[mt]}f(\frac{i}{m})\xi_{i,m},$
then we have the $\mathcal{F}$-stable convergence
\begin{eqnarray}
X_{m}(t)&\underset{m\rightarrow\infty}{\overset{Law}{\longrightarrow}}%
&\int_{0}^{t}f(s)dB_{s},  \notag
\end{eqnarray}
in the Skorohod space $\mathcal{D}[0,T].$
\end{theorem}

We study a stochastic process 
$\int_{0}^{t}u_{s}dY_{s}^{(1)}$, 
where $u$ is a H\"older continuous process of order $\beta > 1-H$. Note that, as $Y^{(1)}$ is H\"older continuous of order $H-\varepsilon$ (a fact that is easily verified by using the integration by parts), the integral can be understood as a Riemann-Stieltjes integral. In particular, the process is well-defined.

We are now ready to state our first main result that provides us the uniform strong consistency.

\begin{theorem}
\label{theorem:Consistency} Suppose that $u=\{u_{t},t\in[0,T]\}$ is an
H\"older continuous stochastic process of order $\beta$ with $\beta > 1-H$, and set
\begin{eqnarray}
\label{eq:process_Z}
Z_{t} &=& \int_{0}^{t}u_{s}dY_{s}^{(1)}.
\end{eqnarray}
Then, as $n$ tends to infinity,
\begin{eqnarray}
n^{2H-1}V_{n}(Z)_{t} &\longrightarrow &
\int_{0}^{t}|u_{s}|^2ds,
\end{eqnarray}
almost surely and uniformly in $t$.
\end{theorem}
Before proving the statement we introduce one auxiliary lemma. 
Let $\upsilon_{Y^{(1)}}$ be the variogram of the process $Y^{(1)}$, i.e.
\begin{eqnarray}
\upsilon_{Y^{(1)}}(t) &=& \frac{1}{2}\mathrm{I\kern-0.16em E}%
\left(Y^{(1)}_{t+s}-Y_{s}^{(1)}\right)^{2}.
\end{eqnarray}
We also use the standard notation $f(t) = o(g(t))$ as $t\rightarrow 0$, if 
$$
\lim_{t\rightarrow 0}\frac{|f(t)|}{|g(t)|} = 0.
$$
\begin{lemma}
\label{lemma:variogram function}The variogram function $\upsilon_{Y^{(1)}}$,
satisfies
\begin{eqnarray}
\upsilon_{Y^{(1)}}(t) &= & \frac{t^{2H}}{2} + o\left(t^{2H}\right) \ \ as \ \ t\rightarrow 0^+.  \notag
\end{eqnarray}
\end{lemma}

\begin{proof}
Proposition 12 of \cite{KS} yields
\begin{eqnarray}  \label{eq: variogram function}
\mathrm{I\kern-0.16em E}\left(Y^{(1)}_{t}-Y_{s}^{(1)}\right)^{2} &=&
2\int_{0}^{t-s}(t-s-x)k_{H}(x)dx, 
\end{eqnarray}
where
\begin{eqnarray}
k_{H}(x) &=& H(2H-1)H^{2H-2}e^{-\frac{(1-H)}{H}x}\left|1-e^{-x/H}%
\right|^{2H-2} .  \notag
\end{eqnarray}
Thus
\begin{eqnarray}
\upsilon_{Y^{(1)}}(t) &=&\frac{1}{2}\mathrm{I\kern-0.16em E}%
\left(Y^{(1)}_{t+s}-Y_{s}^{(1)}\right)^{2}  \notag \\
&=& \int_{0}^{t}(t-x)k_{H}(x)dx  \notag \\
&=&H(2H-1)H^{2H-2} \int_{0}^{t}(t-x)e^{-\frac{(1-H)}{H}x}\left|1-e^{-x/H}%
\right|^{2H-2}dx  \notag \\
&=&H(2H-1)H^{2H-2} \int_{0}^{t}ue^{-\frac{(1-H)}{H}(t-u)}\left|1-e^{-\frac{1%
}{H}((t-u)}\right|^{2H-2}dx  \notag \\
&=&H(2H-1)H^{2H-2} \int_{e^{-t/H}}^{1}(t+\log v^H)v^{1-H}(1-v)^{2H-2}\frac{H%
}{v}dv  \notag \\
&=&H^{2}(2H-1)H^{2H-2}
\left(t\int_{e^{-t/H}}^{1}v^{-H}(1-v)^{2H-2}dv+\int_{e^{-t/H}}^{1}\log
(v^H)v^{-H}(1-v)^{2H-2}dv\right).  \notag
\end{eqnarray}
Write
\begin{eqnarray}
t\int_{e^{-t/H}}^{1}v^{-H}(1-v)^{2H-2}dv &=& t^{2H}\frac{%
\int_{e^{-t/H}}^{1}v^{-H}(1-v)^{2H-2}dv}{t^{2H-1}}  \notag \\
&=:&t^{2H}a_{t}.  \notag
\end{eqnarray}
By L'Hôpital Theorem
\begin{eqnarray}
\lim_{t\rightarrow 0^+}a_{t} &=& \lim_{t\rightarrow 0^+} \frac{1/H
e^{-t/H}e^{t}(1-e^{-t/H})^{2H-2}}{(2H-1)t^{2H-2}}  \notag \\
&=& \lim_{t\rightarrow 0^+} \frac{1}{H(2H-1)}\left(\frac{e^{t/2H}-e^{-t/2H}}{%
t}\right)^{2H-2}  \notag \\
&=&\frac{1}{H(2H-1)H^{2H-2}}.  \notag
\end{eqnarray}
Similarly, writing
\begin{eqnarray}
\int_{e^{-t/H}}^{1}\log v^Hv^{-H}(1-v)^{2H-2}dv &=& t^{2H}\frac{%
\int_{e^{-t/H}}^{1}\log v^Hv^{-H}(1-v)^{2H-2}dv}{t^{2H}}  \notag \\
&=:& t^{2H}b_{t},  \notag
\end{eqnarray}
we obtain by L'Hôpital Theorem that
\begin{eqnarray}
\lim_{t\rightarrow 0^+}b_{t} &=& \lim_{t\rightarrow 0^+} \frac{1/H
e^{-t/H}(-t)e^{t}(1-e^{-t/H})^{2H-2}}{2Ht^{2H-1}}  \notag \\
&=& \lim_{t\rightarrow 0^+} -\frac{e^{-t/H}e^{t}(1-e^{-t/H})^{2H-2}}{%
H(2H)t^{2H-2}}  \notag \\
&=&\frac{-1}{H(2H)H^{2H-2}}.  \notag
\end{eqnarray}
Thus, as $t\rightarrow 0^+$,
\begin{eqnarray*}
a_{t}+b_{t} &\rightarrow& \frac{1}{2H^{2}(2H-1)H^{2H-2}}.
\end{eqnarray*}
As 
$$
\upsilon_{Y^{(1)}}(t) = H^{2}(2H-1)H^{2H-2}t^{2H}(a_{t}+b_{t}),
$$
we get
$$
\frac{\upsilon_{Y^{(1)}}(t)}{t^{2H}} \rightarrow \frac12.
$$
Equivalently, this means
$$
\frac{\upsilon_{Y^{(1)}}(t)}{t^{2H}} - \frac12 = o(1),
$$
as $t\rightarrow 0^+$ which finishes the proof. 
\end{proof}

\begin{proof}[Proof of Theorem \ref{theorem:Consistency}]
For $t\in [0,T]$ and an integer $n$, we denote by $[tn]$ the largest integer that is at most $tn$. Let now $m\geq n$. We have
\begin{eqnarray}
&&m^{-1+2H}V_{n}(Z)_{t}-\int_{0}^{t}|u_{s}|^2ds  \notag \\
&=&m^{2H-1}\sum_{j=1}^{[mt]}\left(\left|%
\int_{(j-1)/m}^{j/m}u_{s}dY^{(1)}_{s}\right|^{2}-\left|u_{\frac{j-1}{m}%
}(Y^{(1)}_{\frac{j}{m}}-Y^{(1)}_{\frac{j-1}{m}})\right|^2\right)  \notag \\
&+&m^{2H-1}\left(\sum_{j=1}^{[mt]}\left|u_{\frac{j-1}{m}}\left(Y^{(1)}_{%
\frac{j}{m}}-Y^{(1)}_{\frac{j-1}{m}}\right)\right|^2-\sum_{i=1}^{[nt]}%
\left|u_{\frac{i-1}{n}}\right|^2\sum_{j\in I_{n}(i)}\left|Y^{(1)}_{\frac{j}{m%
}}-Y^{(1)}_{\frac{j-1}{m}}\right|^2\right)  \notag \\
&+&m^{2H-1}\sum_{i=1}^{[nt]}\left|u_{\frac{i-1}{n}}\right|^2\sum_{j\in
I_{n}(i)}\left|Y^{(1)}_{\frac{j}{m}}-Y^{(1)}_{\frac{j-1}{m}%
}\right|^2-n^{-1}\sum_{i=1}^{[nt]}\left|u_{\frac{i-1}{n}}\right|^2  \notag \\
&+&\left(n^{-1}\sum_{i=1}^{[nt]}\left|u_{\frac{i-1}{n}}\right|^2-%
\int_{0}^{t}|u_{s}|^2ds\right)  \notag \\
&=& A_{t}^{(m)}+B_{t}^{(n,m)}+C_{t}^{(n,m)}+D_{t}^{(n)},  \notag
\end{eqnarray}
where
\begin{eqnarray}
I_{n}(i)&=&\left\{j:\frac{j}{m}\in\left(\frac{i-1}{n},\frac{i}{n}\right], \ \ 1\leqslant i\leqslant [nt]. \right\} \notag
\end{eqnarray}
The idea of the proof is that we first let $m\rightarrow \infty$ and then $n\rightarrow \infty$, and we show that each of the terms $A_t^{(m)}, B_t^{(n,m)}, C_t^{(n,m)}$, and $D_t^{(n)}$ converges to zero almost surely, and uniformly in $t$.

Let us begin with the term $C_t^{(n,m)}$. We have 
\begin{eqnarray}
\parallel C^{(n,m)}\parallel_{\infty}
&\leqslant&\sum_{i=1}^{[nT]}\left|u_{\frac{i-1}{n}}\right|^{2}\left|m^{2H-1}%
\sum_{j\in I_{n}(i)}\left| Y^{(1)}_{\frac{j}{m}}-Y^{(1)}_{\frac{j-1}{m}%
}\right|^2-n^{-1}\right|.  \notag
\end{eqnarray}
As we first let $m\rightarrow \infty$, it suffices to show that, for a fixed $n$, we have 
\begin{eqnarray}
\left|m^{2H-1}\sum_{j\in I_{n}(i)}\left| Y^{(1)}_{\frac{j}{m}}-Y^{(1)}_{%
\frac{j-1}{m}}\right|^2-n^{-1}\right| \rightarrow 0.  \notag
\end{eqnarray}
For this we apply Lemma \ref{lemma:lemma of Lauri} and Theorem \ref{theorem:BE_bound_QV}. First note that it follows from (\ref{eq: variogram function}) that
\begin{eqnarray}
\mathrm{I\kern-0.16em E}\left(Y^{(1)}_{\frac{j}{m}}-Y^{(1)}_{\frac{j-1}{m}}\right)^2 &=&
2\int_{0}^{1/m}(1/m-x)k_{H}(x)dx  \notag \\
&=& \mathrm{I\kern-0.16em E}\left(Y^{(1)}_{\frac{1}{m}}-Y^{(1)}_{0}\right)^2.  \notag
\end{eqnarray}
Combining this with Lemma \ref{lemma:variogram function}, we get
\begin{eqnarray*}
m^{-1+2H}\sum_{j\in I_{n}(i)} \mathrm{I\kern-0.16em E} \left[%
\left(Y^{(1)}_{\frac{j}{m}}-Y^{(1)}_{\frac{j-1}{m}}\right)^2\right] &=&\mathrm{I\kern-0.16em
E}\left(Y^{(1)}_{\frac{1}{m}}-Y^{(1)}_{0}\right)^2 m^{-1+2H}\sum_{j\in I_{n}(i)}1  \notag \\
& = & m^{-2H}m^{-1+2H}\frac{m}{n} + o(1)\rightarrow \frac{1}{n}
\end{eqnarray*}
as $m\longrightarrow \infty$. On the other hand, the covariance of the increments is given by 
(see \cite{KS})
\begin{eqnarray}
\mathrm{I\kern-0.16em E}%
\left((Y^{(1)}_{t}-Y^{(1)}_{s})(Y^{(1)}_{v}-Y^{(1)}_{u})\right)&=&%
\int_{s}^{t}\int_{u}^{v}r_{H}(w,z)dwdz,  \notag
\end{eqnarray}
where $r_{H}(u,v)$ is a symmetric kernel given by
\begin{eqnarray}
r_{H}(u,v)&=&H(2H-1)H^{2(H-1)} \frac{e^{-(1-H)(u-v)/H}}{\left|1-e^{-(u-v)/H}%
\right|^{2(1-H)}}.  \notag
\end{eqnarray}
Now it is straightforward to check that the assumption \eqref{eq:derivative_cond_needed} is satisfied. Thus, thanks to Theorem \ref{theorem:BE_bound_QV}, $\parallel C_t^{(n,m)}\parallel_\infty \rightarrow 0$ almost surely.

Consider next the term $A_t^{(m)}$. We have  
\begin{eqnarray}
|A_{t}^{(m)}|&\leqslant&m^{2H-1}\sum_{j=1}^{[mt]}\left|\left|%
\int_{(j-1)/m}^{j/m}u_{s}dY^{(1)}_{s}\right|^{2}-\left|u_{\frac{j-1}{m}%
}(Y^{(1)}_{\frac{j}{m}}- Y^{(1)}_{\frac{j-1}{m}})\right|^2\right|.  \notag
\end{eqnarray}
We will use the following inequality (see \cite{HNZ}), valid for any $p\geq 0$ and any $x,y\in \mathbb{R}$,
\begin{eqnarray}  \label{eq:inequality}
\left||x|^p-|y|^p\right| &\leqslant& (p\vee 1)2^{(p-2)^{+}}\left[%
|x-y|^p+|y|^{(p-1)^{+}}|x-y|^{(p\wedge1)}\right].
\end{eqnarray}
This implies
\begin{eqnarray}
|A_{t}^{(m)}|
&\leqslant&2m^{-1+2H}\sum_{j=1}^{[mt]}\left|%
\int_{(j-1)/m}^{j/m}u_{s}dY^{(1)}_{s}-u_{\frac{j-1}{m}%
}(Y^{(1)}_{\frac{j}{m}}-Y^{(1)}_{\frac{j-1}{m}})\right|^2  \notag \\
&+&2m^{-1+pH}\sum_{j=1}^{[mt]}\left|u_{\frac{j-1}{m}}\left(Y^{(1)}_{\frac{j}{m}}-
Y^{(1)}_{\frac{j-1}{m}}\right)\right|\left|%
\int_{(j-1)/m}^{j/m}u_{s}dY^{(1)}_{s}-u_{\frac{j-1}{m}%
}\left(Y^{(1)}_{\frac{j}{m}}-Y^{(1)}_{\frac{j-1}{m}}\right)\right|  \notag \\
&=:&E_{(m)}(t)+R_{(m)}(t),  \notag
\end{eqnarray}
where
\begin{eqnarray}
E_{(m)}(t)&=&2m^{-1+2H}\sum_{j=1}^{[mt]}\left|%
\int_{(j-1)/m}^{j/m}u_{s}dY^{(1)}_{s}-u_{\frac{j-1}{m}}\left(Y^{(1)}_{\frac{j}{m}}-
Y^{(1)}_{\frac{j-1}{m}}\right)\right|^2,  \notag \\
R_{(m)}(t)&=&2m^{-1+2H}\sum_{j=1}^{[mt]}\left|u_{\frac{j-1}{m}%
}\left(Y^{(1)}_{\frac{j}{m}}- Y^{(1)}_{\frac{j-1}{m}}\right)\right|\left|%
\int_{(j-1)/m}^{j/m}u_{s}dY^{(1)}_{s}-u_{\frac{j-1}{m}%
}\left(Y^{(1)}_{\frac{j}{m}}-Y^{(1)}_{\frac{j-1}{m}}\right)\right|.  \notag
\end{eqnarray}
For the term $E_{(m)}(t)$ we observe, by applying Young inequality \eqref{Young inequality}, that 
\begin{eqnarray}
|E_{(m)}(t)| &\leqslant&c_{H,\beta,\varepsilon} m^{2H-1}
\sum_{j=1}^{[mT]}\left|var_{\frac{1}{\beta}}(u;\mathcal{I}_{m}(j))var_{1/(H-%
\varepsilon)}(Y^{(1)};\mathcal{I}_{m}(j))\right|^2,  \notag
\end{eqnarray}
where $\mathcal{I}_{m}(j)=\left(\frac{j-1}{m},\frac{j}{m}\right]$%
and $0<\varepsilon<H$.\newline
By \eqref{eq:holder_var} we have
\begin{eqnarray}
var_{\frac{1}{\beta}}(u,\mathcal{I}_{m}(j)) &\leqslant& m^{-\beta}\|u\|_{\beta}  \notag
\end{eqnarray}
and
\begin{eqnarray}
var_{1/(H-\varepsilon)}(Y^{(1)},\mathcal{I}_{m}(j)) &\leqslant&
m^{-(H-\varepsilon)}\|Y^{(1)}\|_{H-\varepsilon}.  \notag
\end{eqnarray}
Thus
\begin{eqnarray}
\|E_{(m)}\|_{\infty} &\leqslant&c_{H,\beta,\varepsilon}m^{2H-1-2\beta} \|u\|^2_{\beta}
\sum_{j=1}^{[mT]}\left|var_{1/(H-\varepsilon)}(Y^{(1)};\mathcal{I}%
_{m}(j))\right|^2,  \notag \\
&\leqslant&Tc_{H,\beta,\varepsilon}m^{2H-1-2\beta-2(H-\varepsilon)+1} \|u\|^2_{\beta}\|Y^{(1)}\|^2_{(H-\varepsilon)}\notag \\
&\leqslant&Tc_{H,\beta,\varepsilon}m^{2(\varepsilon-\beta)}
\|u\|^2_{\beta}\|Y^{(1)}\|^2_{(H-\varepsilon)}.  \notag
\end{eqnarray}
As we can choose $\varepsilon< \beta$, this implies that $\lim_{m\rightarrow \infty}\|E_{(m)}\|_{\infty}=0$ almost surely. Similarly, 
we can apply \eqref{Young inequality} to the term $R_{(m)}(t)$ to get
\begin{eqnarray}
|R_{(m)}(t)|&\leqslant&c_{H,\beta,\varepsilon}m^{-1+2H}\sum_{j=1}^{[mT]}\left|u_{\frac{%
j-1}{m}}(Y_{\frac{j}{m}}^{(1)}-Y_{\frac{j-1}{m}}^{(1)})\right|%
\left|var_{\frac{1}{\beta}}(u,\mathcal{I}_{m}(j))var_{1/(H-\varepsilon)}(Y^{(1)},\mathcal{I%
}_{m}(j))\right|  \notag \\
&\leqslant&2c_{H,\beta,\varepsilon}m^{-1+2H-\beta-(H-\varepsilon)}\parallel
u\parallel_{\beta}\parallel
Y^{(1)}\parallel_{H-\varepsilon}\sum_{j=1}^{[mT]}\left|u_{\frac{j-1}{m}}(Y_{%
\frac{j}{m}}^{(1)}-Y_{\frac{j-1}{m}}^{(1)})\right|  \notag \\
&\leqslant&2c_{H,\beta,\varepsilon}m^{-1+H-\beta+\varepsilon}\parallel
u\parallel_{\beta}\parallel
Y^{(1)}\parallel_{H-\varepsilon}\parallel
u\parallel_{\infty}\sum_{j=1}^{[mT]}\left|var_{1/(H-\varepsilon)}(Y^{(1)},%
\mathcal{I}_{m}(j))\right|  \notag \\
&\leqslant&Tc_{H,\beta,\varepsilon}\parallel u\parallel_{\beta}\parallel
Y^{(1)}\parallel^2_{H-\varepsilon}\parallel
u\parallel_{\infty}m^{-\beta+2\varepsilon}.  \notag
\end{eqnarray}
Hence, for $\varepsilon < \frac{\beta}{2}$, we get $\parallel R_{(m)}\parallel_{\infty} \rightarrow 0$ almost surely, and consequently, $\parallel A^{(m)}\parallel_{\infty} \rightarrow 0$ almost surely as $m\rightarrow\infty$.

It remains to study the terms $D^{(n)}_t$ and $B^{(n,m)}_t$. For the term $D^{(n)}_t$ we first observe that for any $s\in \left[\frac{i-1}{n},\frac{i}{n}\right]$, we have
$$
||u_{\frac{i-1}{n}}|^2-|u_s|^2| \leq 2\|u\|_{\infty}\|u\|_{\beta} n^{-\beta}.
$$
Thus we can estimate
\begin{eqnarray}
|D_{t}^{(n)}|&=&\left|n^{-1}\sum_{i=1}^{[nt]}|u_{\frac{i-1}{n}}|^2-%
\int_{0}^{t}|u_s|^2ds\right|  \notag \\
&=&\left|\sum_{i=1}^{[nt]}\int_{(i-1)/n}^{i/n}(|u_{\frac{i-1}{n}}|^2-|u_s|^2)ds + \int_{[nt]/n}^t|u_s|^2ds \right|  \notag \\
&\leqslant&\sum_{i=1}^{[nt]}\int_{(i-1)/n}^{i/n}\left||u_{\frac{i-1}{n}}|^2-|u_{s}|^2\right|ds + \int_{[nt]/n}^t|u_s|^2ds \notag\\
&\leqslant&2T\|u\|_{\infty}\|u\|_{\beta} n^{-\beta} + \|u\|_{\infty}|t-[nt]/n|\notag\\
&\leqslant&2T\|u\|_{\infty}\|u\|_{\beta} n^{-\beta} + \|u\|_{\infty}n^{-1}.\notag
\end{eqnarray}
This implies that also $\parallel D^{(n)}\parallel_{\infty}\rightarrow 0$ almost surely as $n\rightarrow 0$. It remains to study the term $B_{t}^{(n,m)}$. First note that, by the definition of $I_n(i)$, we have
$$
\sum_{j=1}^{[mt]}\left|u_{\frac{j-1}{m}}\left(Y^{(1)}_{%
\frac{j}{m}}-Y^{(1)}_{\frac{j-1}{m}}\right)\right|^2 = \sum_{i=1}^{[nt]}\sum_{j\in I_n(i)}\left|u_{\frac{j-1}{m}}\left(Y^{(1)}_{%
\frac{j}{m}}-Y^{(1)}_{\frac{j-1}{m}}\right)\right|^2.
$$
Together with the fact that 
$$
|u_{\frac{j-1}{m}} - u_{\frac{i-1}{n}}|^2 \leq 4\parallel u\parallel_{\infty}\parallel u\parallel_\beta n^{-\beta}
$$
as $\frac{j}{m} \in \left(\frac{i-1}{n},\frac{i}{n}\right]$, 
this gives us
\begin{eqnarray}
|B_{t}^{(n,m)}|&=&\left|m^{2H-1}\left(\sum_{j=1}^{[mt]}\left|u_{\frac{j-1}{m}}\left(Y^{(1)}_{%
\frac{j}{m}}-Y^{(1)}_{\frac{j-1}{m}}\right)\right|^2-\sum_{i=1}^{[nt]}%
\left|u_{\frac{i-1}{n}}\right|^2\sum_{j\in I_{n}(i)}\left|Y^{(1)}_{\frac{j}{m%
}}-Y^{(1)}_{\frac{j-1}{m}}\right|^2\right) \right|  \notag \\
&\leq &\sum_{i=1}^{[nt]}\sum_{j\in I_n(i)}|u_{\frac{j-1}{m}} - u_{\frac{i-1}{n}}|^2\left|Y^{(1)}_{%
\frac{j}{m}}-Y^{(1)}_{\frac{j-1}{m}}\right|^2 \notag\\
&\leq & 4\parallel u\parallel_{\infty}\parallel u\parallel_\beta n^{-\beta}\sum_{i=1}^{[nt]}\sum_{j\in I_n(i)}\left|Y^{(1)}_{%
\frac{j}{m}}-Y^{(1)}_{\frac{j-1}{m}}\right|^2. \notag
\end{eqnarray}
Here
$$
\sum_{j\in I_n(i)}\left|Y^{(1)}_{%
\frac{j}{m}}-Y^{(1)}_{\frac{j-1}{m}}\right|^2 \rightarrow n^{-1}
$$
almost surely, and thus
$$
\sum_{i=1}^{[nt]}\sum_{j\in I_n(i)}\left|Y^{(1)}_{%
\frac{j}{m}}-Y^{(1)}_{\frac{j-1}{m}}\right|^2 \rightarrow t
$$
almost surely. This implies that $\parallel B^{(n,m)}\parallel_{\infty}\rightarrow 0$ which completes the proof.
\end{proof}

\begin{theorem}
\label{theorem:distribution asymptotic}Suppose that $u=\{u_{t},t\in[0,T]\}$ is an
H\"older continuous stochastic process of order $\beta$ with $\beta > \max\left(1-H,\frac12\right)$, and measurable with respect to $\mathcal{F}_T^{H}$, where $\mathcal{F}_T^H$ is the filtration generated by $B^H$, or equivalently, by $Y^{(1)}$. Set
\begin{eqnarray}
Z_{t} &=& \int_{0}^{t}u_{s}dY_{s}^{(1)}.
\end{eqnarray}
Then, as $n$ tends to infinity,
\begin{eqnarray*}
n^{2H-1/2}V_{n}(Z)_{t}-\sqrt{n}\int_{0}^{t}|u_{s}|^2ds%
 &\overset{\mathcal{L}}{\rightarrow}&
\int_{0}^{t}|u_{s}|^2dW_{s}
\end{eqnarray*}
$\mathcal{F}_T^H$-stably in the space $\mathcal{D}([0,T]^2)$, where $W=\{W_{t},t\in[0,T]\}$ is a Brownian motion
independent of $\mathcal{F}_{T}^H$.
\end{theorem}

\begin{proof}
As in the proof of Theorem \ref{theorem:Consistency}, we make a decomposition 
\begin{eqnarray}
n^{2H-1/2}V_{n}(Z)_{t}-\sqrt{n}\int_{0}^{t}|u_{s}|^2 ds &=:&
A_{t}^{(n)}+B_{t}^{(n)}+C_{t}^{(n)},  \notag
\end{eqnarray}
where
\begin{eqnarray*}
A_{t}^{(n)}
&=&n^{2H-1/2}\sum_{i=1}^{[nt]}\left(\left|%
\int_{(i-1)/n}^{i/n}u_{s}dY^{(1)}_{s}\right|^{2}-\left|u_{\frac{i}{n}%
}(Y^{(1)}_{\frac{i}{n}}-Y^{(1)}_{\frac{i-1}{n}})\right|^2\right),  \notag \\
\notag \\
B_{t}^{(n)}&=&n^{2H-1/2}\sum_{i=1}^{[nt]}\left|u_{\frac{i}{n}}\left(Y^{(1)}_{%
\frac{i}{n}}-Y^{(1)}_{\frac{i-1}{n}}\right)\right|^2-\frac{1}{\sqrt{n}}%
\sum_{i=1}^{[nt]}\left|u_{\frac{i}{n}}\right|^2,  \notag \\
C_{t}^{(n)}&=&\frac{1}{\sqrt{n}}\sum_{i=1}^{[nt]}\left|u_{\frac{i}{n}%
}\right|^2-\sqrt{n}\int_{0}^{t}|u_{s}|^2ds.  \notag
\end{eqnarray*}
Using $\beta > \frac12$ and treating the terms $A_t^{(n)}$ and $C_t^{(n)}$ as the terms $A_t^{(m)}$ and $D_t^{(n)}$ in the proof of Theorem \ref{theorem:Consistency}, we obtain 
$$
\parallel A^{(n)}\parallel_\infty +\parallel C^{(n)}\parallel_\infty \rightarrow 0
$$
almost surely. Consider next the term $B_{t}^{(n)}$.
We set
\begin{eqnarray*}
\xi_{i,n} &=& n^{2H-1/2}\left|Y^{(1)}_{\frac{i}{n}} -Y^{(1)}_{\frac{(i-1)}{n}%
}\right|^2-\frac{1}{\sqrt{n}}  \notag
\end{eqnarray*}
so that 
$$
B_{t}^{(n)}= \sum_{i=1}^{[nt]}|u_{i/n}|^{2}\xi_{i,n}.
$$ 
In order to complete the proof, we need to verify hypotheses (H1) and (H2) of Theorem \ref{theorem: Hypotheses}. However, the hypothesis (H1) follows from the general convergence result established in \cite{LL}. Indeed, the process $Y^{(1)}$ can be written as an integral with respect to the fBm as 
\begin{eqnarray}
Y_{t}^{(1)} &=& \int_{H}^{a_t}e^{-a_s^{-1}}dB^{H}_{s}  \notag
\end{eqnarray}
which follows easily from a change of variable. For hypothesis (H2), we obtain by Proposition 4.2 in \cite{Taqqu} together with direct computations that, for any $1\leqslant j <k\leqslant [nT]$, we have
\begin{eqnarray}
\mathrm{I\kern-0.16em E}\left(\left|\sum_{i=j+1}^{k}\xi_{i,n}\right|^4%
\right) &=&\frac{1}{n^2} \mathrm{I\kern-0.16em E}\left(\left|%
\sum_{i=j+1}^{k}n^{2H}\left|Y^{(1)}_{\frac{i}{n}} -Y^{(1)}_{\frac{(i-1)}{n}%
}\right|^2-1\right|^4\right)  \notag \\
&=&\frac{1}{n^2} \mathrm{I\kern-0.16em E}\left(\left|\sum_{i=j+1}^{k}G(%
\Delta Y_{i}^{(1)})\right|^4\right)  \notag \\
&\leqslant&\frac{(k-j)^2}{n^2}\left(\sum_{i=0}^{\infty}\rho_{Y^{(1)}}^2(i)%
\right)^2  \notag \\
&\leqslant&C\left(\frac{k-j}{n}\right)^2  \notag
\end{eqnarray}
\end{proof}

\section{Application to the estimation of the integrated volatility}
\label{sec:vol}
In this section we apply our main results to the estimation of the integrated volatility $%
\int_{0}^{t}|\sigma_s|^{2}ds$. We consider a generalized fractional Ornstein-Uhlenbeck process of the second kind defined as the solution to the stochastic differential equation
\begin{eqnarray}  \label{eq:SDE}
dX_{t} &=& -\theta X_{t}dt+\sigma_{t}dY^{(1)}_{t} ,
\end{eqnarray}
with some initial condition $X_0 \in \mathbb{R}$. We define the estimator $QV_{n}(X)_{t}$ for the integrated volatility $%
\int_{0}^{t}|\sigma_s|^{2}ds$ as 
\begin{eqnarray}  \label{eq:estimator}
QV_{n}(X)_{t} &=&n^{2H-1}V_{n}(X)_{t}, \ \ t \in [0,T].
\end{eqnarray}
We begin with two simple propositions which allows us to introduce drift to the process defined by \eqref{eq:process_Z}.
\begin{proposition}
\label{proposition:consistency}
Suppose that the assumptions of Theorem \ref{theorem:Consistency} prevail, and let $Y=\{Y_{t},t\in [0,T]\}$ be a stochastic process such that, as $n$ tends to infinity,
\begin{eqnarray}
n^{2H-1}V_{n}(Y)_{t} &\rightarrow& 0  \notag
\end{eqnarray}
almost surely and uniformly in $t$. Then
\begin{eqnarray}
n^{2H-1}V_{n}(Y+Z)_{t}&\underset{n\rightarrow\infty}{\overset{u.c.p}{%
\longrightarrow}}&\int_{0}^{t}|u_s|^2ds.  \notag
\end{eqnarray}
almost surely and uniformly in $t$.
\end{proposition}

\begin{proof}
By observing that 
$$
V_n(Y+Z) = V_n(Y) + V_n(Z) + 2 \sum_{i=1}^{[nt]}\left(Z_{\frac{i}{n}}-Z_{\frac{i-1}{n}}\right)\left(Y_{\frac{i}{n}}-Y_{\frac{i-1}{n}}\right),
$$
a simple application of Cauchy-Schwarz inequality yields
$$
n^{2H-1}\left[V_n(Y+Z) - v_n(Z)\right] \leq n^{2H-1}V_n(Y) + 2\sqrt{n^{2H-1}V_n(Z)}\sqrt{n^{2H-1}V_n(Y)}
$$
from which the claim follows by Theorem \ref{theorem:Consistency}.
\end{proof}

Similarly, we obtain the following result on the weak convergence.
\begin{proposition}
\label{proposition:clt}
Suppose that the assumptions of Theorem \ref{theorem:distribution asymptotic} prevail, and let $Y=\{Y_{t},t\in [0,T]\}$ be a stochastic process such that, as $n$ tends to infinity,
\begin{eqnarray}
n^{2H-1}V_{n}(Y)_{t} &\rightarrow& 0  \notag
\end{eqnarray}
and uniformly in probability. Then
\begin{eqnarray}
n^{2H-\frac{1}{2}}(Y+Z)_{t}-\sqrt{n}\int_{0}^{t}|u_s|^2ds&%
\underset{n\rightarrow\infty}{\overset{Law}{\longrightarrow}}%
&\int_{0}^{t}|u_s|^2dW_{s}  \notag
\end{eqnarray}
$\mathcal{F}_T^H$-stably in $\mathcal{D}([0,T])$, where $W=\{W_{t},t\in [0,T]\}$ is a Brownian motion independent of $\mathcal{F}_T^H$.
\end{proposition}
\begin{proof}
The result follows directly from Theorem \ref{theorem:distribution asymptotic} and Slutsky's theorem together with the above computations.
\end{proof}

Consider now the estimator \eqref{eq:estimator} for the integrated volatility. With the help of Proposition \ref{proposition:consistency} and Proposition \ref{proposition:clt} we obtain the following results.
\begin{theorem}
Suppose that $\sigma_s$ is a H\"older continuous function of order $\beta > 1-H$. Then 
$$
QV_{n}(X)_{t} \longrightarrow \int_{0}^{t}|\sigma_s|^2ds
$$
almost surely and uniformly in $t$. 
\end{theorem}
\begin{proof}
Recall that $X$ satisfies \eqref{eq:SDE}. Thus we have
$$
X_{t}= X_{0}+Y_{t}+\int_{0}^{t}\sigma_{s}dY_{s}^{(1)},
$$
where $Y_{t}=-\theta\int_{0}^{t}X_{s}ds.$ It is straightforward to check that the solution $X$ is bounded on every compact interval. Consequently, the process $Y_t$ is differentiable with bounded derivative, and thus 
$$
V_n(Y) \leq \theta \parallel X\parallel_\infty^2 n^{-1}.
$$
Now the result follows from Proposition \ref{proposition:consistency} and Theorem \ref{theorem:Consistency}.
\end{proof}
\begin{theorem}
Suppose that $\sigma = \{\sigma_{t},t\in[0,T]\}$ is H\"older continuous of order $\beta > \max\left(1-H,\frac12\right)$, and measurable with respect to $\mathcal{F}_T^{B^H}$. Suppose further that $0< H < \frac34$. Then 
\begin{eqnarray}
\sqrt{n}\left(QV_{n}(X)_{t}-\int_{0}^{t}|\sigma_s|^2ds\right)&\underset{%
n\rightarrow\infty}{\overset{Law}{\longrightarrow}}&\int_{0}^{t}|%
\sigma_s|^2dW_{s},  \notag
\end{eqnarray}
$\mathcal{F}_T^H$-stably in the space $\mathcal{D}([0,T]^2)$, where $W=\{W_{t},t\in[0,T]\}$ is a Brownian motion
independent of $\mathcal{F}_{T}^H$.
\end{theorem}

\begin{proof}
Observing that since $0< H < \frac34$, we have, for $Y_{t}=
-\theta\int_{0}^{t}X_{s}ds,$ that
$$
n^{2H-\frac12}V_n(Y) \leq \theta \parallel X\parallel_\infty^2 n^{2H-\frac32} \rightarrow 0.
$$
Thus the result follows directly from Proposition \ref{proposition:clt} and Theorem \ref{theorem:distribution asymptotic}.
\end{proof}

\section{Discussions}
In this article we have studied realized quadratic variation of the process $
\int_{0}^{t}u_{s}dY^{(1)}_{s}$. As our main result, we established almost sure convergence, uniform in time, of the normalised realized quadratic variation. By following our proof, one can also provide stronger mode of convergence for the realized quadratic variation of the process $
\int_{0}^{t}u_{s}dB^H_{s}$, studied in \cite{CNW}. In addition, we have provided stable convergence in law for the whole region $H\in(0,1)$ of the Hurst parameter. By taking account the fact that $Y^{(1)}$ has a short memory, it is not surprising that one can cover also values $H\in\left(\frac34,1\right)$. However, in order to estimate the integrated volatility we posed an additional condition $H<\frac34$. At this point it is not clear whether this condition can be removed without going into higher order differences as in \cite{CNW}, or can our proofs be refined to obtain stable convergence for the estimator of the integrated volatility.


\begin{thebibliography}{99}
\bibitem{BS2002} \emph{O.E. Barndroff-Nielsen and N. Shephard}. \emph{%
Econometric analysis of realized volatility and its use in estimating
stochastic volatility models}. J. Roy. Statist. Soc., Ser. B, 64, 253-280,
2002.

\bibitem{BS2003} \emph{O.E. Barndroff-Nielsen and N. Shephard}. \emph{%
Realized power variation and stochastic volatility models}. Bernoulli, 9,
243-265, 2003.

\bibitem{BS2004a} \emph{O.E. Barndroff-Nielsen and N. Shephard}. \emph{Power
and bipower variation with stochastic volatility and jumps (with discussion)}%
. J. Financial Econometrics, 2, 1-48, 2004.

\bibitem{BS2004b} \emph{O.E. Barndroff-Nielsen and N. Shephard}. \emph{%
Econometric analysis of realised covariation: high frequency covariance,
regression and correlation in financial economics}. Econometrica, 72,
885-925, 2004.

\bibitem{BSal} \emph{O.E. Barndroff-Nielsen, S.E. Graversen, J. Jacod, M.
Podolskij and N. Shephard}. \emph{A central limit theorem for realised power
and bipower variations of continuous semimartingales.} In Y. Kabanov, R.
Liptser and J. Stoyanov (eds), From Stochastic Analysis to Mathematical
Finance: The Shiryaev Festschrift. Berlin: Springer-Verlag, 2006.

\bibitem{CM} \emph{P. Cheridito, H. Kawaguchi and M. Maejima}. \emph{%
Fractional Ornstein-Uhlenbeck processes}. Elecron. J. Probab, 1-14, 2003.

\bibitem{CNP} \emph{J.M. Corcuera, D. Nualart and M. Podolskij}. \emph{%
Asymptotics of weighted random sums}. Communications in Applied and
Industrial Mathematics, ISSN 2038-0909, 2014.

\bibitem{CNW} \emph{J.M. Corcuera, D. Nualart and J.H.C. Woerner}. \emph{%
Power variation of some integral fractional processes}. Bernoulli, 12(4)-
713-735,2006.

\bibitem{HNZ} \emph{Y. Hu , D. Nualart and H. Zhou}. \emph{Parameter
estimation for fractional Ornstein Uhlenbeck processes of general Hurst
parameter}. Submitted, ArXiv: 1703.09372, 2017.

\bibitem{KN} \emph{P. Kloeden and A. Neuenkirch}. \emph{The pathwise
convergence of approximation schemes for stochastic differential equations }%
. LMS J. Comp. Math, 10, 235-253, 2007.

\bibitem{KS} \emph{T. Kaarakka and P. Salminen}. \emph{On Fractional
Ornstein-Uhlenbeck process}. Communications on Stochastic Analysis.
5,121-133, 2011.

\bibitem{LL} \emph{J.R Leon and C. Ludena}, \emph{Stable convergence of
certain functionals of diffusions driven by fBm}. Stochastic Anal. Appl, 22,
289-314, 2004.

\bibitem{Taqqu} \emph{\ M.S Taqqu}, \emph{Law of the iterated logarithm for
sums of non-linear functions of Gaussian variables that exhibit a long range
dependence}. Wahrscheinlichkeitstheorie Verw. Geb., 40, 203-238, 1977.

\bibitem{Lauri Viitasaari} \emph{L. Viitasaari}, \emph{Sufficient and
Necessary Conditions for Limit Theorems for Quadratic Variations of Gaussian
Sequences}. Under revision, ArXiv: 1502.01370, 2015.

\bibitem{W2003} \emph{J.H.C Woerner}, \emph{Variational sums and power
variation: a unifying approach to model selection and estimation in
semimartingale models}. Statist. Decisions, 21, 47-68, 2003.

\bibitem{W2005} \emph{J.H.C Woerner}, \emph{Estimation of integrated
Volatility in Stochastic Volatility Models}. Appl. Stochastic Models
Business Industry, 21, 27-44, 2005.
\end{thebibliography}
\end{document}